\documentclass[reqno]{amsart}

\usepackage{enumitem}
\usepackage{amsmath}
\usepackage{amsfonts}
\usepackage{amsthm}
\usepackage{amssymb}
\usepackage{graphicx}
\usepackage{graphics}
\usepackage{bm}
\usepackage{dsfont}
\usepackage{color}
\usepackage[font=footnotesize]{caption}
\usepackage[all]{xy}
\numberwithin{equation}{section}
\newtheoremstyle{personal}%
{12pt}%      Space above
{12pt}%      Space below
{\slshape}%         Body font
{}%         Indent amount
{\bfseries}% Theorem head font
{.}%        Punctuation after theorem head
{.5em}%     Space after theorem head
{}%         Theorem head spec (can be left empty, meaning "normal")
\theoremstyle{personal}%
\newtheorem{thm}{Theorem}[section]
\newtheorem{cor}[thm]{Corollary}
\newtheorem{lem}[thm]{Lemma}

\newtheorem{prop}[thm]{Proposition}
\newtheorem{quest}[thm]{Question}
\theoremstyle{definition}
\newtheorem{rem}[thm]{Remark}

\newcommand{\fix}{\mathrm{fix}}
\newcommand{\A}{\mathcal{A}}
\newcommand{\M}{\mathcal{M}}

\newcommand{\B}{\mathcal{B}}
\newcommand{\K}{\mathds{K}}
\newcommand{\N}{\mathds{N}}
\newcommand{\Z}{\mathds{Z}}
\newcommand{\R}{\mathds{R}}

\newcommand{\F}{\mathds{F}}
\newcommand{\RP}{\mathds{R}\PP}

\newcommand{\PP}{\mathds{P}}
\newcommand{\C}{\mathds{C}}
\newcommand{\UU}{\mathcal{U}}
\newcommand{\WW}{\mathcal{W}}

\newcommand{\ECC}{\mathrm{ECC}}
\newcommand{\ECH}{\mathrm{ECH}}
\newcommand{\Hom}{\mathrm{H}}

\newcommand{\diff}{\mathrm{d}}

\newcommand{\id}{\mathrm{id}}

\newcommand{\sigmap}{\sigma_{\mathrm{p}}}

\newcommand{\Tan}{\mathrm{T}}

\begin{document}

\title[Contact 3-manifolds all of whose Reeb orbits are closed]{The action spectrum characterizes closed contact 3-manifolds all of whose Reeb orbits are closed}

\author{Daniel Cristofaro-Gardiner}
\address{Daniel Cristofaro-Gardiner\newline\indent Department of Mathematics, University of California, Santa Cruz\newline\indent
1156 High Street, Santa Cruz, CA 95064, USA}
\email{dcristof@ucsc.edu}

\author{Marco Mazzucchelli}
\address{Marco Mazzucchelli\newline\indent CNRS, \'Ecole Normale Sup\'erieure de Lyon, UMPA\newline\indent  46 all\'ee d'Italie, 69364 Lyon Cedex 07, France\newline\indent and\newline\indent Mathematical Sciences Research Institute\newline\indent 17 Gauss Way, Berkeley, CA 94720, USA}
\email{marco.mazzucchelli@ens-lyon.fr}

\date{January 30, 2019. \emph{Revised}: May 19, 2019.}

\begin{abstract}
A classical theorem due to Wadsley implies that, on a connected contact manifold all of whose Reeb orbits are closed, there is a common period for the Reeb orbits. In this paper we show that, for any Reeb flow on a closed connected 3-manifold, the following conditions are actually equivalent: (1)~every Reeb orbit is closed; (2)~all closed Reeb orbits have a common period; (3)~the action spectrum has rank 1. We also show that, on a fixed closed connected 3-manifold, a contact form with an action spectrum of rank 1 is determined (up to pull-back by diffeomorphisms) by the set of minimal periods of its closed Reeb orbits.
\end{abstract}

\maketitle

\section{Introduction}
\label{s:intro}

A much studied problem in Riemannian geometry asks to what degree a Riemannian manifold is determined by its length spectrum, that is, the set of lengths of its closed geodesics.   It is known that the length spectrum does not in general recover the metric, but more refined conjectures and results exist, see for example \cite{Croke:1990tk, Otal:1990yv, Guillarmou:2018qo} and references therein.

In contact geometry, an analogous question exists, but little is known.  Recall that a contact form on a closed $(2n+1)$-manifold $Y$ is a 1-form $\lambda$ such that $\lambda\wedge(\diff\lambda)^n$ is a volume form on $Y$.  The kernel of $\diff\lambda$ is then generated by a unique vector field $R_\lambda$ such that $\lambda(R_\lambda)\equiv1$, called the Reeb vector field, which defines a Reeb flow $\phi_\lambda^t:Y\to Y$.  A Reeb orbit $\gamma:\R\to Y$, $\gamma(t)=\phi_\lambda^t(z)$ is said to be closed if it is $\tau$-periodic for some $\tau>0$, i.e.\ $\gamma(t)=\gamma(t+\tau)$ for all $t\in\R$. As usual, the minimal period of a closed Reeb orbit $\gamma$ is the minimal $\tau>0$ such that $\gamma$ is $\tau$-periodic; the multiples of such $\tau$ will be simply called periods of $\gamma$. The subset $\sigma(Y,\lambda)\subset(0,\infty)$ consisting of the (not necessarily minimal) periods of the closed Reeb orbits of $\phi_\lambda^t$ is the \textbf{action spectrum} of the contact manifold, whereas its subset $\sigmap(Y,\lambda)\subsetneq\sigma(Y,\lambda)$ consisting of the minimal periods of the closed Reeb orbits of $\phi_\lambda^t$ is the \textbf{prime action spectrum}.  One can now ask to what degree we can characterize $\lambda$ from its action and prime action spectra. In the present note we establish some positive results in dimension~3.

\subsection{Setup and main results}
A contact form $\lambda$ is called \textbf{Besse} when every orbit of its Reeb flow is closed. Our first result states that one can recognize whether a contact form on a closed connected $3$-manifold is Besse from its action spectrum.  We define the \textbf{rank} of the action spectrum $\sigma(Y,\lambda)$ to be the rank of the $\mathbb{Z}$-submodule of $\mathbb{R}$ that it generates (this is the same as the rank of the submodule generated by the prime action spectrum $\sigmap(Y,\lambda)$). In particular, $\sigma(Y,\lambda)$ has rank 1 if and only if it is contained in a subset of the form $\{nT\ |\ n\in\N\}$ for some $T>0$.

\begin{thm}
\label{t:Besse}
Let $(Y,\lambda)$ be a closed connected 3-manifold equipped with a contact form. The following conditions are equivalent:
\begin{itemize}[topsep=3pt]
\item[$(\mathrm{i})$] The contact manifold $(Y,\lambda)$ is Besse.
\item[$(\mathrm{ii})$] The closed orbits of the Reeb flow $\phi_\lambda^t$ have a common period, i.e.\ there is $\tau>0$ such that $\tau/\tau'$ is an integer for all $\tau'\in \sigmap(Y,\lambda)$.
\item[$(\mathrm{iii})$] The action spectrum $\sigma(Y,\lambda)$ has rank 1.
\end{itemize}
\end{thm}

The fact that the closed Reeb orbits of a Besse contact manifold admit a common period, and thus that the action spectrum has rank 1, is a consequence of a classical theorem due to  Wadsley \cite{Wadsley:1975sp}, together with Sullivan's remark \cite{Sullivan:1978bl} that Reeb flows are geodesible. The novelty, here, is the reverse implication, namely that the fact that the action spectrum has rank 1 forces a contact form to be Besse.

A contact form $\lambda$ is called {\bf Zoll} when it is Besse and its closed Reeb orbits have the same minimal period. Namely, when there exists $\tau>0$ such that $\phi_\lambda^\tau=\id$, and for all $t\in(0,\tau)$ the map $\phi_\lambda^t$ has no fixed points. Theorem~\ref{t:Besse} has the following immediate corollary.   

\begin{cor}
\label{c:Zoll}
A closed contact 3-manifold is Zoll if and only if its closed Reeb orbits have the same minimal period.
\hfill\qed
\end{cor}

\begin{rem}
\label{r:answer}
In \cite[Question 1.2]{Mazzucchelli:2018ek}, the second author and Suhr asked whether a reversible contact form on the unit cotangent bundle of any surface must be Zoll if all its closed Reeb orbits have the same minimal period.  (The motivation for this comes from the connection between the contact geometry of the unit cotangent bundle and Riemannian and Finsler geometry, which we say more about below.)  Corollary~\ref{c:Zoll} answers this in the affirmative, and without the reversibility requirement on the contact form.
\end{rem}

To the best of the authors' knowledge, for general higher dimensional closed contact manifolds it is not known whether the Besse or the Zoll properties can be read off from the action spectra.
\begin{quest}
Let $(Y,\lambda)$ be a closed contact manifold of dimension $n\geq5$. If all its closed Reeb orbits have the same minimal period, is $\lambda$ necessarily Zoll? $Y$ is connected and the action spectrum $\sigma(Y,\lambda)$ has rank 1, is $\lambda$ necessarily Besse?
\end{quest}

By Theorem~\ref{t:Besse}, from the action spectrum one can determine whether or not a contact form on a closed connected 3-manifold is Besse. However, it is not possible to recover the contact form (up to pull-back by diffeomorphisms) from the action spectrum in the Besse case.  For example, the standard 1-form $\lambda_{\mathrm{std}} = \frac{1}{2} \sum_{i=1,2} \big(x_i dy_i - y_i dx_i\big)$ on $\R^4$ restricts as a contact form to the boundary of any symplectic ellipsoid
\[ E(a,b) := \left \lbrace \frac{ \pi |z_1|^2}{a} + \frac{ \pi |z_2|^2}{b} \le 1 \right \rbrace \subset \C^2 = \R^4.\]
Its Reeb flow always has two closed orbits of minimal period $a$ and $b$. 
When $b/a$ is rational, the contact form is Besse and the other closed Reeb orbits have minimal period $\mathrm{lcm}(a, b)$.  Thus, $\partial E(1,1)$ and $\partial E(1,2)$ have the same action spectrum, but their contact forms cannot be diffeomorphic.  We can distinguish these ellipsoids, however, through the prime action spectrum. Indeed, our next theorem states that, in the Besse case, the prime action spectrum  always determines the contact form up to pull-back by diffeomorphisms.

\begin{thm}\label{t:classification}
Let $Y$ be a closed connected 3-manifold, and $\lambda_1,\lambda_2$ two Besse contact forms on $Y$. Then $\sigmap(Y,\lambda_1)=\sigmap(Y,\lambda_2)$ if and only if there exists a diffeomorphism $\psi:Y\to Y$ such that $\psi^*\lambda_2=\lambda_1$.
\end{thm}

In the Zoll case, Theorem~\ref{t:classification} was proved by Abbondandolo et al.\ \cite{Abbondandolo:2017xz, Abbondandolo:2018fb} for $S^3$ and $\mathrm{SO}(3)$, and by Benedetti-Kang \cite[Lemma~2.3]{Benedetti:2018ys} for general $S^1$-bundles over closed surfaces.

\begin{rem}
\label{r:nottrue}
Theorems~\ref{t:classification} and~\ref{t:Besse} in combination provide a spectral recognition result: the contact form of a fixed closed connected 3-manifold can be recovered from its prime action spectrum, provided its action spectrum has rank 1.  In higher rank, however, the same cannot in general be done.  For example, in \cite[Theorem 1.2]{Albers:2018rq} Albers-Geiges-Zehmisch construct a contact form $\lambda$ on $S^3$ whose Reeb flow has a dense orbit and only two closed orbits. The minimal periods $ a$ and $b$ of these two orbits are rationally independent. So, the action spectrum $\sigma(S^3,\lambda)$ is the same as $\sigma(\partial E(a,b),\lambda_{\mathrm{std}})$, but there is no diffeomorphism $\psi:S^3\to \partial E(a,b)$ such that $\lambda=\psi^*\lambda_{\mathrm{std}}$.
\end{rem}

\subsection{Finsler geometry}

Theorem~\ref{t:Besse} and Corollary~\ref{c:Zoll} apply in particular to Finsler geodesic flows of 2-spheres. We recall that a Finsler metric on a closed manifold $M$ is a continuous function $F:\Tan M\to[0,\infty)$ that is smooth outside the $0$-section, fiberwise positively homogeneous of degree $1$, and such that $\partial_{vv}F^2(x,v)$ is positive definite at every point $(x,v)$ outside the $0$-section. The Finsler metric $F$ is reversible when $F(x,v)=F(x,-v)$ for all $(x,v)\in\Tan M$, and Riemannian when it is of the form $F(x,v)=g_x(v,v)^{1/2}$ for some Riemannian metric $g$ on $M$. The geodesic flow of $(M,F)$ is precisely the Reeb flow of $(SM,\lambda)$, where $\pi:SM\to M$ is the $F$-unit tangent bundle of $M$ and $\lambda$ is the Liouville form $\lambda_{(x,v)}(w)=\partial_vF(x,v)\diff\pi(x,v)w$.
The action spectrum $\sigma(SM,\lambda)$ is the usual length spectrum of $(M,F)$, and is denoted by $\sigma(M,F)$. The Finsler metric $F$ is Besse or Zoll if the associated Liouville form $\lambda$ is so.

In \cite{Mazzucchelli:2018ek}, the second author and Suhr established (a slightly stronger version of) Corollary~\ref{c:Zoll} for geodesic flows of Riemannian 2-spheres. Theorem~\ref{t:Besse} actually implies the following more general corollaries for Finsler geodesic flows of surfaces.

\begin{cor}
\label{c:finsler}
Let $(M,F)$ be a closed connected orientable Finsler surface. The length spectrum $\sigma(M,F)$ has rank 1 if and only if $M=S^2$ and $F$ is Besse. Moreover, if $F$ is reversible, the length spectrum $\sigma(M,F)$ has rank 1 if and only if $M=S^2$ and $F$ is Zoll.
\end{cor}

\begin{rem}
The reversibility assumption in the second part of this statement is essential. Indeed, certain of the so-called Katok's metrics on the 2-sphere \cite{Ziller:1983rw} are examples of non-reversible Finsler metrics that are Besse but not Zoll.
\end{rem}

\begin{proof}[Proof of Corollary~\ref{c:finsler}]
The fact that the length spectrum of a Finsler closed connected surface has rank 1 if and only if the metric is Besse follows from Theorem~\ref{t:Besse}. A theorem due to Frauenfelder-Labrousse-Schlenk \cite{Frauenfelder:2015sn}, which extends the classical Bott-Samelson Theorem \cite{Bott:1954aa, Samelson:1963aa} from Riemannian geometry, implies that $F$ can be Besse only if the fundamental group of $M$ is finite and the integral cohomology ring of the universal cover of $M$ agrees with that of a compact rank-one symmetric space. The only closed orientable surface $M$ with these properties is $S^2$. Finally, a Besse reversible Finsler metric on $S^2$ is Zoll according to a theorem of Frauenfelder-Lange-Suhr \cite{Frauenfelder:2016ud}, which generalizes the classical Riemannian result of Gromoll-Grove \cite{Gromoll:1981kl}.
\end{proof}

\begin{cor}
\label{c:RP2}
Let $(M,F)$ be a closed connected non-orientable Finsler surface. The length spectrum $\sigma(M,F)$ has rank 1 if and only if $M=\RP^2$ and $F$ is Besse. Moreover, if $F$ is Riemannian, the length spectrum $\sigma(M,F)$ has rank 1 if and only if $M=\RP^2$ and $F$ is Riemannian with constant curvature (in particular, $F$ is Zoll).
\end{cor}

\begin{proof}
Let $M'$ be the orientation double cover of $M$, and $F':\Tan M'\to[0,\infty)$ the lift of $F$. By Corollary~\ref{c:finsler}, $F'$ is Besse if and only if $\sigma(M',F')$ has rank 1 and $M'=S^2$. Notice that $M'=S^2$ if and only if $M=\RP^2$. The length spectra satisfy $\sigma(M',F')\subseteq\sigma(M,F)$ and $2\sigma(M,F)\subseteq\sigma(M',F')$; in particular,  $\sigma(M',F')$ has rank 1 if and only if the same is true for $\sigma(M,F)$. Moreover, $F'$ is Besse if and only if the same if true for $F$. 
This proves the first part of the statement. Finally, a Riemannian metric on $\RP^2$ is Besse if and only if it has constant curvature, according to a theorem of Pries \cite{Pries:2009aa}.
\end{proof}

\subsection{Relationship with previous work and organization of the paper}

A corollary of Theorem~\ref{t:Besse} is that any contact form on a closed 3-manifold has at least two distinct closed embedded Reeb orbits.  This was previously proved by the first author and Hutchings \cite{Cristofaro-Gardiner:2016rp} using embedded contact homology.  Our proof of Theorem~\ref{t:Besse} uses a similar method; the main difference here is a strengthening of one of the key lemmas in that paper, see our Lemma~\ref{l:Besse_spectral} below.  In contrast, the proof of Theorem~\ref{t:classification} does not require embedded contact homology, but instead makes use of the classification of Seifert fibered spaces, in combination with a Moser trick in Lemma~\ref{l:Moser_trick}.

The paper is organized as follows.  In Section~\ref{s:ECH_1} we provide the needed background on embedded contact homology. In Section~\ref{s:ECH_2}, we prove our main Theorem~\ref{t:Besse}; in the proof, we will need a slightly stronger version of the bumpy contact form theorem, which we state and prove in Appendix~\ref{a:bumpy}. In Section~\ref{s:Seifert}, after introducing the needed preliminaries on Seifert fibered spaces, we prove Theorem~\ref{t:classification}.

\subsection*{Acknowledgments}
The authors are grateful to the anonymous referee for her/his careful reading of the manuscript, and for pointing out the statement of Corollary~\ref{c:RP2}.
Daniel Cristofaro-Gardiner is partially supported by the National Science Foundation under Grant No.~1711976. Marco Mazzucchelli is partially supported by the National Science Foundation under Grant No.~DMS-1440140 while in residence at the Mathematical Sciences Research Institute in Berkeley, California, during the Fall 2018 semester.

\section{Background on Embedded Contact Homology}
\label{s:ECH_1}

In this section we will recall the essential features of embedded contact homology that will be needed in order to prove Theorem~\ref{t:Besse}. The interested reader will find a detailed account and precise references in Hutchings' survey \cite{Hutchings:2014qf}.

\subsection{The chain complex}
Let $(Y,\xi)$ be a closed connected oriented contact manifold of dimension 3. Throughout this paper, the contact distribution $\xi\subset\Tan Y$ is assumed to be cooriented, and as usual we will call a 1-form $\lambda$ on $Y$ a supporting contact form of $\xi$ when $\ker(\lambda)=\xi$ and $\lambda$ induces the orientation of $\Tan Y/\xi$. The 2-form $\diff\lambda$ will then induce an orientation on $\xi$. The contact form $\lambda$ is called bumpy when, for each $\tau>0$ and $z\in\fix(\phi_\lambda^\tau)$, 1 is not an eigenvalue of the linearized Poicar\'e map $\diff\phi_\lambda^\tau(z)|_\xi$.
We will write the symplectization of our contact manifold as $(\R\times Y,\diff(e^s\lambda))$, where $s$ is the variable on $\R$. The embedded contact homology group $\ECH(Y)$ is a topological invariant obtained as the homology of a chain complex $\big(\ECC(Y,\lambda),\partial_{Y,\lambda,J}\big)$, where $\lambda$ is a bumpy supporting contact form of $(Y,\xi)$, and $J$ is an almost complex structure on $(\R\times Y,\diff(e^s\lambda))$ such that $JR_\lambda=\tfrac\partial{\partial s}$, $J\xi=\xi$, $\diff\lambda(v,Jv)>0$ for each non-zero $v\in\xi$, and $J$ is chosen generically in order to satisfy suitable technical assumptions. The chain group $\ECC(Y,\lambda)$ is the $\Z_2$-vector space freely generated by finite sets of pairs $\{(m_i,\gamma_i)\ |\ i=1,...,k \}$, where $k\in\N$, the $\gamma_i$ are distinct simple closed orbits of the Reeb flow $\phi_\lambda^t$, and $m_i$ is a positive integer required to be equal to 1 if $\gamma_i$ is hyperbolic. Here, by ``simple'' we mean that the closed Reeb orbits $\gamma_i$ are viewed as maps of the form $\gamma_i:\R/\tau_i\Z\to Y$, where $\tau_i>0$ is the minimal period of $\gamma_i$. Two simple closed Reeb orbits $\gamma_i,\gamma_j$ are distinct if they are not of the form $\gamma_i=\gamma_j(\cdot+s)$ for any $s>0$.
 The definition of the differential $\partial_{Y,\lambda,J}$ involves counting certain $J$-holomorphic curves in the symplectization of $(Y,\lambda)$, but will not be needed in the present paper.

\subsection{The $U$ map}
The embedded contact homology comes equipped with an endomorphism
\begin{align*}
U:\ECH(Y)\to\ECH(Y)
\end{align*}
defined as follows. Let $\bm\gamma=\{(m_i,\gamma_i)\ |\ i=1,...,k\}$ and $\bm\zeta=\{(n_i,\zeta_i)\ |\ i=1,...,l\}$ be two chains in $\ECC(Y,\lambda)$.  Let $(\Sigma,j)$ be a punctured Riemann surface, and $u:\Sigma\to\R\times Y$ a $J$-holomorphic curve that is asymptotic as a current to $\sum_i m_i\gamma_i$ and $\sum_i n_i\zeta_i$ as $s\to\infty$ and $s\to-\infty$ respectively. 
We denote by $\M(J,\bm\gamma,\bm\zeta)$ the space of such $J$-holomorphic curves modulo equivalence as currents. 
Notice that, for every $u\in\M(J,\bm\gamma,\bm\zeta)$, we have
\begin{align*}
\int_\Sigma u^*\diff\lambda = \sum_{i=1}^k m_i\A_\lambda(\gamma_i) - \sum_{i=1}^{l} n_i\A_\lambda(\zeta_i).
\end{align*}
Here, $\A_\lambda$ denotes the contact action
\begin{align*}
 \A_\lambda(\gamma)=\int_\gamma \lambda.
\end{align*}
If $\gamma$ is a simple closed Reeb orbit, $\A_\lambda(\gamma)$ is simply its minimal period. To every $u\in\M(J,\bm\gamma,\bm\zeta)$ there is an associated integer which is called the ECH-index, and whose definition will not be needed in the present paper. For a given $z\in Y$, we denote by $\M_{2,z}(J,\bm\gamma,\bm\zeta)\subset\M(J,\bm\gamma,\bm\zeta)$ the subset of those $u:\Sigma\to\R\times Y$ having ECH-index 2 and whose image $u(\Sigma)$ passes through $(0,z)$. The condition on the ECH index implies that, if $J$ is chosen generically, then $\M_{2,z}(J,\bm\gamma,\bm\zeta)$ is a finite set. The endomorphism 
\begin{align*}
U_z:\ECC(Y,\lambda)\to \ECC(Y,\lambda),
\qquad
U_z(\bm\gamma)= \!\!\!\! \sum_{\bm\zeta\in\ECC(Y,\lambda)} \!\!\!\! (\# \M_{2,z}(J,\bm\gamma,\bm\zeta)\ \mathrm{mod}\ 2)\,\bm\zeta
\end{align*}
turns out to be a chain map that induces the endomorphism $U$ in embedded contact homology. Notice that $U_z$ depends on the chosen point $z$, on the contact form $\lambda$, and on the almost complex structure $J$, whereas $U$ is a topological invariant of $Y$.

\subsection{Spectral invariants}

Given a supporting contact form $\lambda$ on a closed contact 3-manifold $(Y,\xi)$, we denote by $\Sigma(Y,\lambda)\subset(0,\infty)$ the set of real numbers that are finite sums of elements in the action spectrum $\sigma(Y,\lambda)$, i.e.
\begin{align*}
\Sigma(Y,\lambda)=\big\{\tau_1+...+\tau_k\ \big|\ k\geq 1,\ \tau_i\in\sigma(Y,\lambda)\quad \forall i=1,...,k \big\}.
\end{align*}
The chain complex $(\ECC(Y,\lambda),\partial_{Y,\lambda,J})$ can be filtered by means of the action as follows. For each $\tau>0$, let $\ECC^\tau(Y,\lambda)$ be the vector subspace of $\ECC(Y,\lambda)$ generated by those $\bm\gamma=\{(m_i,\gamma_i)\ |\ i=1,...,k\}$ such that
\begin{align*}
\A_\lambda(\bm\gamma):=\sum_{i=1}^k m_i\A_\lambda(\gamma_i) \leq\tau.
\end{align*}
Since the boundary map $\partial_{Y,\lambda,J}$ does not increase the action, $\big(\ECC^\tau(Y,\lambda),\partial_{Y,\lambda,J}\big)$ is a subcomplex of $\big(\ECC(Y,\lambda),\partial_{Y,\lambda,J}\big)$, whose homology is denoted by $\ECH^\tau(Y,\lambda)$. As the notation suggests, this latter group turns out to be independent of the almost complex structure $J$.  There is an inclusion induced map
\[ \iota^\tau: \ECH^\tau(Y,\lambda) \to \ECH(Y) .\]
Each non-zero $\sigma\in \ECH(Y)$ defines a spectral invariant $c_\sigma(Y,\lambda)\in\Sigma(Y,\lambda)$ as follows. If $\lambda$ is bumpy, then $c_\sigma(Y,\lambda)$ is the minimal $\tau>0$ such that $\sigma$ admits a representative in $\ECC^\tau(Y,\lambda)$, in other words such that $\sigma$ is in the image of the map $\iota^\tau$. 
If $\lambda$ is not bumpy, we can choose a sequence of smooth functions $b_n:Y\to\R,$ $C^0$-converging to zero and such that each contact form $e^{b_n}\lambda$ is bumpy (see Proposition~\ref{p:bumpy}); in this case, the sequence $c_\sigma(Y,e^{b_n}\lambda)$ converges and the spectral invariant $c_\sigma(Y,\lambda)$ is defined as its limit, i.e.
\begin{align}
\label{e:lim_c}
c_\sigma(Y,\lambda) = \lim_{n\to\infty} c_\sigma(Y,e^{b_n}\lambda). 
\end{align}
The following statement due to the first author and Hutchings provides the only property of spectral invariants needed in this paper.  It is an application of the Volume Property for the ECH spectrum proved in \cite{Cristofaro-Gardiner:2015wa}.

\begin{lem}[Cor.~2.2 in \cite{Cristofaro-Gardiner:2016rp}]
\label{l:CGH}
There exists a sequence $\{\sigma_k\ |\ k\in\N\}$ of non-zero elements in $\ECH(Y)$ such that $U\sigma_{k+1}=\sigma_k$ and  $c_{\sigma_k}(Y,\lambda)/k\to0$ as $k\to\infty$ for each supporting contact form $\lambda$ of $(Y,\xi)$.
\hfill\qed
\end{lem}

\section{ECH-spectral characterization of Besse contact forms}
\label{s:ECH_2}

The following statement, which improves \cite[Lemma~3.1(b)]{Cristofaro-Gardiner:2016rp} while following a similar logic, is the main ingredient for proving Theorem~\ref{t:Besse}.  
\begin{lem}
\label{l:Besse_spectral}
Let $(Y,\lambda)$ be a closed connected contact 3-manifold equipped with a contact form.
If $c_\sigma(Y,\lambda)=c_{U\sigma}(Y,\lambda)$ for some $\sigma\in\ECH(Y)$ with $U\sigma\neq 0$, then $(Y,\lambda)$ is Besse.
\end{lem}

\begin{proof}
Assume that $(Y,\lambda)$ is not Besse, so that there exists $z\in Y$ such that $\phi_\lambda^t(z)\neq z$ for all $t\neq0$. We set $c:=c_\sigma(Y,\lambda)$, and fix an arbitrary real number $\tau>c$. Let $\Sigma\subset Y$ be an embedded compact ball of codimension 1 containing $z$ in its interior and such that $\Tan_z\Sigma=\xi_z$, where $\xi=\ker(\lambda)$ is the contact distribution. Up to shrinking $\Sigma$ around $z$, the map
\begin{align*}
\psi:[-\tau/2,\tau/2]\times\Sigma\to Y,\qquad \psi(t,w)=\phi_\lambda^t(w)
\end{align*}
is a diffeomorphism onto its image $K:=\psi([-\tau/2,\tau/2]\times\Sigma)$. Namely, $K$ is a flow box for the Reeb flow $\phi_\lambda^t$ containing orbits of length $\tau$.

We fix an almost complex structure $J$ on the symplectization $(\R\times Y,\diff(e^s\lambda))$ such that $JR_\lambda=\tfrac{\partial}{\partial s}$, $J\xi=\xi$, and $\diff\lambda(v,Jv)>0$ for all non-zero $v\in\xi$. By Proposition~\ref{p:bumpy}, there exists a sequence $b_n\in C^\infty(Y)$ such that $b_n|_{K}\equiv0$, $b_n\to0$ in $C^0$ and $\lambda_n:=e^{b_n}\lambda$ is a bumpy contact form. Since $\lambda_n\equiv \lambda$ on $K$, this latter set is also a flow box for the Reeb flows $\phi_{\lambda_n}^t$. In particular, none of the closed orbits of $\phi_{\lambda_n}^t$ with minimal period at most $\tau$ intersects $K$. Therefore, we can choose an almost complex structure $J_n$ on the symplectization $(\R\times Y,\diff(e^{s}\lambda_n))$ such that $J_n\equiv J$ on $\R\times K$, and $J_n$ is sufficiently generic to define the differential of the complex $\big(\ECC^\tau(Y,\lambda_n),\partial_{Y,\lambda_n,J_n} \big)$ and the endomorphism $U_z:\ECC^\tau(Y,\lambda_n)\to \ECC^\tau(Y,\lambda_n)$.

We consider an arbitrary cycle $\bm\gamma_n\in\ECC^\tau(Y,\lambda_n)$ such that $\sigma=\iota^\tau([\bm\gamma_n])$ and $c_\sigma(Y,\lambda_n)=\A_{\lambda_n}(\bm\gamma_n)$. Equation~\eqref{e:lim_c} implies that $\A_{\lambda_n}(\bm\gamma_n)\to c_\sigma(Y,\lambda)$ as $n\to\infty$. In order to conclude the proof, we need to show that there exists $\delta>0$ such that
\begin{align*}
\A_{\lambda_n}(\bm\gamma_n)-\A_{\lambda_n}(U_z\bm\gamma_n)\geq\delta,
\qquad\forall n\in\N.
\end{align*}
Indeed, this implies that 
\begin{align*}
c_{U\sigma}(Y,\lambda) 
 = 
\!\lim_{n\to\infty}\! c_{U\sigma}(Y,\lambda_n) \leq   
\!\lim_{n\to\infty}\! \A_{\lambda_n}(U_z\bm\gamma_n)
\leq
\!\lim_{n\to\infty}\!  \A_{\lambda_n}(\bm\gamma_n)-\delta =
c_{\sigma}(Y,\lambda) - \delta.
\end{align*}

Assume by contradiction that 
\begin{align*}
\liminf_{n\to\infty} \big(\A_{\lambda_n}(\bm\gamma_n) -  \A_{\lambda_n}(U_z\bm\gamma_n) \big) = 0.
\end{align*}
Up to extracting a subsequence, we can actually assume that 
\begin{align}
\label{e:zero}
\lim_{n\to\infty} \big(\A_{\lambda_n}(\bm\gamma_n) -  \A_{\lambda_n}(U_z\bm\gamma_n) \big) = 0.
\end{align}
We choose, for each $n\in\N$, a $J_n$-holomorphic curve $u_n:\Sigma_n\to\R\times Y$ in the moduli space $\M_{2,z}(J_n,\bm\gamma_n,U_z\bm\gamma_n)$. We set $C_n:=u_n(\Sigma_n)$, and from now on we will not distinguish between the map $u_n$ and its image $C_n$. Notice that
\begin{align}
\label{e:energy}
\int_{C_n} \diff\lambda_n = \A_{\lambda_n}(\bm\gamma_n) -  \A_{\lambda_n}(U_z\bm\gamma_n),
\end{align}
and in particular this quantity is uniformly bounded in $n$. Since $J_n\equiv J$ on $\R\times K$, the intersections $C_n\cap([-1,1]\times K)$ are $J$-holomorphic curves. Since $\diff\lambda_n=\diff\lambda$ is non-negative on $C_n\cap([-1,1]\times K)$, Equations~\eqref{e:zero} and~\eqref{e:energy} imply that 
\begin{align}
\label{e:zero_Cn}
\lim_{n\to\infty}
\int_{C_n\cap([-1,1]\times K)}  \diff\lambda = 0,
\end{align}
and that this integral is uniformly bounded in $n$. Let $s_0\in[-2,-1]$ and $s_1\in[1,2]$ be such that $u_n$ is transverse to $\{s_0,s_1\}\times Y$. Since both $\diff(e^s\lambda_n)$ and $\diff\lambda_n$ are non-negative on $C_n$ by the conditions on $J_n$, we have the uniform bound
\begin{align*}
\int_{C_n \cap([-1,1]\times K)} \diff(e^s\lambda)
& \leq
\int_{C_n\cap([s_0,s_1]\times Y)} \diff(e^s\lambda_n)\\
&
= 
e^{s_1} \int_{C_n\cap (\{s_1\}\times Y)} \lambda_n
-
e^{s_0} \int_{C_n\cap (\{s_0\}\times Y)} \lambda_n \\
&
\leq
e^2 \bigg( \int_{C_n\cap (\{s_1\}\times Y)} \lambda_n + \int_{C_n\cap([s_1,\infty)\times Y)}\diff\lambda_n\bigg)\\
& =
e^2 \A_{\lambda_n}(\bm\gamma_n) \leq e^2 c_{\sigma}(Y,\lambda) + 1
\end{align*}
for all $n\in\N$ large enough. We can thus employ a compactness result due to Taubes \cite[Prop.~3.3]{Taubes:1998oz}, in its version \cite[Prop.~3.2]{Cristofaro-Gardiner:2016rp}, and infer that, up to extracting a subsequence, the sequence $C_n\cap([-1,1]\times K)$ converges in the sense of currents to a compact $J$-holomorphic curve $C\subset[-1,1]\times K$ with boundary in $\partial([-1,1]\times K)$, and $(0,z)\in C$. Equation~\eqref{e:zero_Cn} thus implies
\begin{align*}
 \int_{C} \diff\lambda =0,
\end{align*}
and therefore $C$ must have a component of the form
$[-1,1]\times\phi_\lambda^{[-\tau/2,\tau/2]}(z)$.
In particular
\begin{align*}
\int_{C \cap (\{s\}\times K)} \lambda \ge \tau,\qquad \forall s\in[-1,1].
\end{align*}
We fix an arbitrary $\tau'\in(c_\sigma(Y,\lambda),\tau)$.
For each $n\in\N$, we choose a point $s_n\in[-1,1]$ such that $u_n$ is transverse to $\{s_n\}\times Y$, and we orient the intersection using the ``$\R$-direction first" convention. By the conditions on $J_n$, the contact form $\lambda_n$ is non-negative along the oriented 1-manifold $C_n \cap (\{s_n\}\times Y)$. Therefore, since $C_n\cap([-1,1]\times K)\to C$ in the sense of currents, up to removing sufficiently many elements from the sequence $\{C_n\ |\ n\in\N\}$ we have
\begin{align*}
\int_{C_n \cap (\{s_n\}\times Y)} \lambda_n \geq \int_{C_n \cap (\{s_n\}\times K)} \lambda_n \geq \tau',\qquad \forall n\in\N.
\end{align*}
However, if we choose $n$ large enough so that $\A_{\lambda_n}(\bm\gamma_n)<\tau'$, we have
\begin{align*}
\int_{C_n \cap (\{s_n\}\times Y)} \lambda_n
\leq
\int_{C_n \cap (\{s_n\}\times Y)} \lambda_n
+
\int_{C_n \cap ([s_n,\infty)\times Y)} \diff\lambda_n
=
\A_{\lambda_n}(\bm\gamma_n)<\tau',
\end{align*}
which gives a contradiction.
\end{proof}

\begin{proof}[Proof of Theorem~\ref{t:Besse}]
We already know that (i) implies (ii) by Wadsley's theorem \cite{Wadsley:1975sp}. Assume now that our closed connected contact 3-manifold $(Y,\xi=\ker(\lambda))$ satifies (ii). We denote by $\tau>0$ a common period for the closed Reeb orbits. Every closed orbit $\gamma$ of the Reeb flow $\phi_\lambda^t$ has minimal period $\tau/k_\gamma$ for some $k_\gamma\in\N=\{1,2,3,...\}$. Since $Y$ is compact and the Reeb vector field of $(Y,\lambda)$ is nowhere vanishing, there is a uniform lower bound for the minimal periods of the closed orbits of $\phi_\lambda^t$. In particular, the set 
\[\K:=\big\{k_\gamma\ \big|\ \gamma \mbox{ closed orbit of }\phi_\lambda^t \big\}\]
is finite. If we denote by $k\in\N$ a common multiple of the natural numbers in $\K$, we readily see that the period of each closed orbit of the Reeb flow $\phi_\lambda^t$ must be a multiple of $\tau/k$. This implies (iii).

Finally, let us assume that $(Y,\lambda)$ satisfies (iii). By Lemma \ref{l:CGH}, there exists a sequence $\{\sigma_k\ |\ k\in\N\}$ of non-zero elements in $\ECH(Y,\xi,\Gamma)$ 
such that $U\sigma_{k+1}=\sigma_k$ and $c_{\sigma_k}(Y,\lambda)/k\to0$ as $k\to\infty$. If $c_{\sigma_{k+1}}(Y,\lambda)\neq c_{\sigma_k}(Y,\lambda)$ for all $k\in\N$, then $c_{\sigma_{k+1}}(Y,\lambda)\geq c_{\sigma_k}(Y,\lambda)+T$, where $T>0$ is such that $\sigma(Y,\lambda)\subset\{nT\ |\ n\in\N\}$. However, this would imply that
\begin{align*}
\liminf_{k\to\infty} c_{\sigma_k}(Y,\lambda)/k \geq T >0,
\end{align*}
which is a contradiction. Therefore we must have $c_{\sigma_{k+1}}(Y,\lambda)= c_{\sigma_k}(Y,\lambda)$ for some (and indeed for infinitely many) $k\in\N$. By Lemma~\ref{l:Besse_spectral}, we conclude that $(Y,\lambda)$ is Besse.
\end{proof}

Recent results of the second author and Suhr, \cite[Theorem~3.1]{Mazzucchelli:2018ek} and \cite[Theorem~1.2]{Mazzucchelli:2018pb}, provide a min-max characterization of certain Zoll Riemannian manifolds by employing Morse-theoretic spectral invariants for the length and energy functionals on the loop space. In the same spirit, the proof of Theorem~\ref{t:Besse} also provides the following ECH-spectral characterization of Besse contact forms.

\begin{thm}
A closed connected contact 3-manifold $(Y,\lambda)$ is Besse if and only if, for some $\sigma\in\ECH(Y)$ with $U\sigma\neq0$, we have $c_{\sigma}(Y,\lambda)=c_{U\sigma}(Y,\lambda)$.
\hfill\qed
\end{thm}

\section{Besse contact forms and Seifert fibrations}
\label{s:Seifert}

\subsection{The Morse-Bott property}
Let us recall that a closed connected Besse contact manifold $(Y,\lambda)$ of any dimension $2n+1\geq3$ has Morse-Bott closed orbits. By the already mentioned Wadsley's Theorem \cite{Wadsley:1975sp}, there exists a minimal $\tau>0$ such that the Reeb flow satisfies $\phi_\lambda^\tau=\id$. Therefore, each point $z\in Y$ lies on a closed Reeb orbit of minimal period $\tau_z=\tau/\alpha_z$, for some $\alpha_z\in\N$. For each $\alpha\in\N$, we define a compact subset \[K_\alpha:=\fix(\phi_\lambda^{\tau/\alpha})\subset Y.\]
Since the Reeb vector field $R_\lambda$ is nowhere vanishing, there exists a finite subset $\F\subset\N$ such that $K_\alpha\neq\varnothing$ if and only if $\alpha\in\F$.
Let $g_0$ be a Riemannian metric on $Y$ such that $g_0(R_\lambda,\cdot)=\lambda$. Its average
\begin{align*}
 g:=\frac1\tau\int_0^\tau (\phi_\lambda^t)^*g_0\,\diff t
\end{align*}
is a Riemannian metric that still satisfies $g(R_\lambda,\cdot)=\lambda$ and is invariant under the Reeb flow, i.e.\ $(\phi_\lambda^t)^*g=g$. Since $\phi_\lambda^{\tau/\alpha}$ is a $g$-isometry, its fixed-point set $K_\alpha$ is a closed submanifold of $Y$ with tangent spaces \[\Tan_zK_\alpha=\ker(\diff\phi_\lambda^{\tau/\alpha}(z)-\id),\] see \cite[Theorem~5.1]{Kobayashi:1995mo}. The linearized map $\diff\phi_\lambda^{\tau/\alpha}(z)|_{\xi_z}$ is a symplectic endomorphism of the symplectic vector space $(\xi_z,\diff\lambda_z|_{\xi_z})$, where $\xi:=\ker(\lambda)$. Therefore, the eigenvalue $1\in\sigma(\diff\phi_\lambda^{\tau/\alpha}(z)|_{\xi_z})$ has even algebraic multiplicity. Since $\diff\phi_\lambda^{\tau/\alpha}(z)|_{\xi_z}$ is an $\alpha$-th root of the identity, this algebraic multiplicity is equal to the geometric multiplicity $\dim\ker(\diff\phi_\lambda^{\tau/\alpha}(z)|_{\xi_z}-\id)$. This, together with the fact that \[\diff\phi_\lambda^{\tau/\alpha}(z)R_\lambda(z)=R_\lambda(z),\] proves that $\dim(\Tan_zK_\alpha)$ is odd, and thus that $K_\alpha$ is an odd-dimensional closed manifold.

\subsection{Seifert fibrations}
We now assume that our Besse closed connected contact manifold $(Y,\lambda)$ has dimension 3. Therefore, the subsets $K_\alpha$ with $\alpha\in\F\setminus\{1\}$ are finite disjoint unions of embedded circles. If $\F\setminus\{1\}\not=\varnothing$, the complement $Y\setminus K$, where $K:=\cup_{\alpha\in\F\setminus\{1\}} K_\alpha$, is an open Zoll contact manifold. The Reeb flow on $Y$ defines a locally free $\R/\tau\Z$-action on $Y$, whose quotient $\Sigma_g$ can be given the structure of a closed orientable surface of some genus $g\geq0$. The quotient map $\pi:Y\to \Sigma_g$ is not a genuine circle bundle if $(Y,\lambda)$ is not Zoll, but it is still a Seifert fibration. Namely, if $\{x_1,...,x_r\}:=\pi(K)$,  for each $x_j$ there are associated parameters $\alpha_j,\beta_j,\alpha'_j,\beta'_j\in\Z$ with the following properties. The parameter $\alpha_j\geq 1$ is such that $\pi^{-1}(x_j)\subset K_{\alpha_j}$. Therefore, $\pi^{-1}(x_j)$ is a closed Reeb orbit of minimal period $\tau/\alpha_j$.
Both pairs $(\alpha_j,\beta_j)$, $(\alpha_j',\beta_j')$ are coprime, and form an integer matrix with determinant $\alpha_j\beta_j'-\alpha_j'\beta_j=1$. The point $x_j$ possesses a compact disk neighborhood $D_j\subset\Sigma_g$ that we identify with the unit ball in the complex plane, and there is a diffeomorphism $\psi_j:D_j\times S^1\to\pi^{-1}(D_j)$ such that
\begin{align*}
\pi\circ\psi_j(\rho z_1 ,z_2)=\rho z_1^{\alpha_j} z_2^{\alpha_j'},
\qquad\forall \rho\in[0,1],\ z_1,z_2\in S^1.
\end{align*}
Here and in the following, $S^1$ denotes the unit circle in the complex plane $\C$.  
The Reeb flow induced on $D_j\times S^1$ has the form
\begin{align*}
\psi_j^{-1}\circ\phi_\lambda^t\circ\psi_j(\rho z_1,z_2)=(\rho z_1e^{-i2\pi\alpha_j't/\tau},z_2e^{i2\pi\alpha_jt/\tau}).
\end{align*}
The restriction $\pi:Y\setminus K\to\Sigma_g\setminus\{x_1,...,x_r\}$ is a trivial $S^1$-bundle, that is, there is a diffeomorphism $\psi:\Sigma_g\setminus\{x_1,...,x_r\}\times S^1\to Y\setminus K$ such that $\pi\circ\psi(z_1,z_2)=z_1$. The Reeb flow induced on $\Sigma_g\setminus\{x_1,...,x_r\} \times S^1$ is simply
\begin{align*}
\psi^{-1}\circ\phi_\lambda^t\circ\psi(z_1,z_2)=(z_1,z_2e^{i2\pi t/\tau}).
\end{align*}
We orient $\Sigma_g$ by means of a 2-form $\omega$ on $\Sigma_g\setminus\{x_1,...,x_r\}$ such that $\pi^*\omega=\diff\lambda|_{Y\setminus K}$, and we orient the fibers of $\pi$ by means of the Reeb vector field $R_\lambda$, so that the diffeomorphisms $\psi|_{\{x\}\times S^1}:\{x\}\times S^1\to\pi^{-1}(x)$ are orientation preserving. We introduce the oriented circles in the torus $T_j:=\pi^{-1}(\partial D_j)$
\begin{align*}
M_j&:=\psi_j(\partial D_j\times\{1\}),&
L_j&:=\psi_j(\{1\}\times S^1),\\
M_j'&:=\psi(\partial D_j\times\{1\}),&
L_j'&:=\psi(\{x\}\times S^1),
\end{align*}
where $x$ is any point in $\partial D_j$. In the homology group $\Hom_1(T_j;\Z)$, we have
\begin{align*}
[M_j] = \alpha_j[M_j']+\beta_j[L_j'],
\qquad
[L_j] = \alpha_j'[M_j']+\beta_j'[L_j'].
\end{align*}
The integers in the tuple $(g;\alpha_1,\beta_1,...,\alpha_r,\beta_r)$ are the so-called Seifert invariants of the Seifert fibration $\pi:Y\to\Sigma_g$, and every $(\alpha_j,\beta_j)$ is called a Seifert pair. 
We stress that the concept of Seifert fibration is more general than the one presented here (for instance it allows for non-orientable total spaces and non-orientable base surfaces), but will not be needed in its full generality for the application to Besse contact forms. In this paper, all Seifert fibrations are implicitly assumed to be of the above type, and in particular with total space and base surface both closed and orientable.

A Seifert fibration can be described by different Seifert invariants tuples, but nevertheless these invariants determine the Seifert fibration completely. More precisely, given two Seifert fibrations $\pi_i:Y_i\to\Sigma_{g_i}$, $i=1,2$, there exist an orientation preserving diffeomorphism $F:Y_1\to Y_2$ and a diffeomorphism $f:\Sigma_{g_1}\to\Sigma_{g_2}$ such that $\pi_2\circ F=f\circ\pi_1$ if and only if the two Seifert fibrations can be described by the same Seifert invariants tuple. 
A theorem due to Raymond \cite{Raymond:1968sf}
(see also \cite[Theorem~2.1]{Jankins:1983zm}) implies that the isomorphism classes of Seifert fibrations are the same as the isomorphism classes of effective $S^1$-actions on $3$-manifolds. This readily implies the following statement in our setting.

\begin{thm}
\label{t:Raymond}
For $i=1,2$, let $(Y_i,\lambda_i)$ be a Besse closed connected contact 3-manifold oriented via the volume form $\lambda_i\wedge\diff\lambda_i$ and whose Reeb orbits have minimal common period $\tau_i$. Then, there exists an orientation preserving diffeomorphism $\psi:Y_1\to Y_2$ such that $\psi\circ\phi_{\lambda_1}^{\tau_1 t}\circ\psi^{-1}=\phi_{\lambda_2}^{\tau_2 t}$ for all $t\in\R$ if and only if $(Y_1,\lambda_1)$ and $(Y_2,\lambda_2)$ have the same Seifert invariants in normal form (up to permutation of the pairs).
\hfill\qed
\end{thm}

A particular case of a result due to Lisca-Mati\'c \cite{Lisca:2004oz} provides a constraint on the Seifert invariants of a Seifert fibration associated to a Besse contact form.
\begin{thm}[Prop.~3.1 in \cite{Lisca:2004oz}]
\label{t:lisca_matic}
The Seifert invariants $(g;\alpha_1,\beta_1,...,\alpha_r,\beta_r)$ of any Besse closed connected contact 3-manifold satisfy
$\tfrac{\beta_1}{\alpha_1}+...+\tfrac{\beta_r}{\alpha_r}>0$.
\hfill\qed
\end{thm}

The Seifert fibrations are classified. In particular, a result due to Orlik-Vogt-Zieschang \cite{Orlik:1967ad} (see also \cite[Section~1]{Geiges:2018zt}) implies that a given closed connected orientable 3-manifold $Y$ admits at most one Seifert fibration structure (up to Seifert fibration isomorphism possibly reversing the orientation of the total space), unless $Y$ is a prism manifold, a single Euclidean manifold, or a lens space. Every manifold that is of prism or single Euclidean type admits two non-isomorphic Seifert fibration structures, one of which projects onto a non-orientable surface. By applying this together with Lisca-Mati\'c's Theorem~\ref{t:lisca_matic}, we obtain the following uniqueness result for Besse contact forms.

\begin{lem}
\label{l:classification_Seifert}
Let $Y$ be a closed connected 3-manifold not homeomorphic to a lens space, and $\lambda_1,\lambda_2$ two Besse contact forms on $Y$ whose Reeb orbits have minimal common periods $\tau_1,\tau_2$ respectively. Then, there exists a diffeomorphism $\psi:Y\to Y$ such that $\psi\circ\phi_{\lambda_1}^{\tau_1 t}\circ\psi^{-1}=\phi_{\lambda_2}^{\tau_2 t}$ for all $t\in\R$, and the volume forms $\psi^*(\lambda_2\wedge\diff\lambda_2)$ and $\lambda_1\wedge\diff\lambda_1$ induce the same orientation on $Y$.
\end{lem}

\begin{proof}
Let $\pi_i:Y\to\Sigma_{g_i}$ be the Seifert fibration defined by the Besse contact form $\lambda_i$. Since $\Sigma_{g_i}$ is orientable and the total space $Y$ is not homeomorphic to a lens space, the above mentioned result of Orlik-Vogt-Zieschang \cite{Orlik:1967ad} implies that there exist diffeomorphisms $F:Y\to Y$ and $f:\Sigma_{g_1}\to\Sigma_{g_2}$ such that $\pi_2\circ F=f\circ\pi_1$. The lemma now follows from Theorem~\ref{t:Raymond} once we prove that $\lambda_1\wedge\diff
\lambda_1$ and $F^*(\lambda_2\wedge\diff
\lambda_2)$ define the same orientation on $Y$.

Let us assume by contradiction that $\lambda_1\wedge\diff
\lambda_1$ and $F^*(\lambda_2\wedge\diff
\lambda_2)$ define opposite orientations on $Y$. If $(g_1;\alpha_1,\beta_1,...,\alpha_r,\beta_r)$ are  Seifert invariants for $\pi_1:Y\to\Sigma_{g_1}$,  Lisca-Mati\'c's Theorem~\ref{t:lisca_matic} implies that
\begin{align}
\label{e:positive_euler}
\tfrac{\beta_1}{\alpha_1}+...+\tfrac{\beta_r}{\alpha_r}>0.
\end{align}
Since $\lambda_1\wedge\diff
\lambda_1$ and $F^*(\lambda_2\wedge\diff
\lambda_2)$ define opposite orientations, the Seifert fibration $\pi_2:Y\to \Sigma_{g_2}$ has Seifert invariants $(g_1;\alpha_1,-\beta_1,...,\alpha_r,-\beta_r)$, and Lisca-Mati\'c's Theorem~\ref{t:lisca_matic} would imply \[-\tfrac{\beta_1}{\alpha_1}-...-\tfrac{\beta_r}{\alpha_r}>0,\] contradicting~\eqref{e:positive_euler}.
\end{proof}

The classification of Seifert fibrations on lens spaces has been recently carried out by Geiges-Lange \cite{Geiges:2018zt}. We summarize their results that we will need as follows. We recall that, for $p$ and $q$ coprime integers and $p>0$, the lens space $L(p,q)$ is the quotient of the unit 3-sphere $S^3\subset\C^2$ under the $\Z/p\Z$-action generated by $(z_1,z_2)\mapsto(e^{i2\pi/p}z_1,e^{i2\pi q/p}z_2)$. When $p$ is not positive, the lens spaces are defined by $L(p,q):=L(-p,-q)$ and $L(0,1):=S^2\times S^1$. If $\pi:L(p,q)\to\Sigma_g$ is a Seifert fibration, then the base surface $\Sigma_g$ is either $S^2$ or $\R P^2$. Since $\Sigma_g$ is orientable whenever the Seifert fibration is defined by a Besse contact form, in this section we will only consider Seifert fibrations of lens spaces over $S^2$.

\begin{thm}[Prop.~4.6--4.8 and Th.~4.10 in \cite{Geiges:2018zt}]
\label{t:geiges_lange}
$ $
\begin{itemize}[topsep=3pt]
\item[$(\mathrm{i})$] Any Seifert fibration $\pi:L(0,1)\to S^2$ has Seifert invariants $(0;\alpha,\beta,\alpha,-\beta)$, where $\alpha$ and $\beta$ are coprime integers such that $\alpha>0$ and $\beta\geq0$.

\item[$(\mathrm{ii})$] If $p>0$, any Seifert fibration $\pi:L(p,q)\to S^2$ with at most one singular fiber has Seifert invariants $(0;\alpha,\beta)$, where $\beta=p$, $\alpha\neq0$, and $\alpha\equiv q$ or $\alpha q\equiv1$ mod $p$. 

\item[$(\mathrm{iii})$] There exist functions $b_1:\Z^4\to\Z$ and $b_2:\Z^4\to\Z$ such that any Seifert fibration $\pi:L(p,q)\to S^2$ with $p>0$ has Seifert invariants $(0;\alpha_1,\beta_1,\alpha_2,\beta_2)$ satisfying $\beta_1=b_1(p,q,\alpha_1,\alpha_2)$, $\beta_2=b_2(p,q,\alpha_1,\alpha_2)$, and the greatest common divisor $\gcd(\alpha_1,\alpha_2)$ divides $p$.
\hfill\qed
\end{itemize} 
\end{thm}

\subsection{Classification of Besse contact 3-manifolds}

The following is the last ingredient needed for proving Theorem~\ref{t:classification}.

\begin{lem}
\label{l:Moser_trick}
For $i=0,1$, let $(Y_i,\lambda_i)$ be a closed contact 3-manifold equipped with a contact form and oriented by means of the volume form $\lambda_i\wedge\diff\lambda_i$. If there exists an orientation preserving diffeomorphism $\psi_0:Y_0\to Y_1$ such that $\diff\psi_0(z)R_{\lambda_0}(z)=R_{\lambda_1}(\psi(z))$ for all $z\in Y_0$, then $\psi_0$ can be isotoped to a diffeomorphism $\psi_1:Y_0\to Y_1$ such that $\psi_1^*\lambda_1=\lambda_0$.
\end{lem}

\begin{proof}
By pulling back the contact form $\lambda_1$ by means of $\psi_0$, we can assume without loss of generality that $Y_0=Y_1=:Y$, $\psi_0=\id$, $R_{\lambda_0}=R_{\lambda_1}$, and both volume forms $\lambda_0\wedge\diff\lambda_0$ and $\lambda_1\wedge\diff\lambda_1$ define the same orientation on $Y$. For each $t\in[0,1]$, the convex combination $\lambda_t:=t\lambda_1+(1-t)\lambda_0$ is a contact form. Indeed, consider any oriented basis of a tangent space of $Y$ of the form $R_{\lambda_0}(z),v,w$. Since $R_{\lambda_0}=R_{\lambda_1}$, notice that
\begin{align*}
\lambda_i\wedge\diff\lambda_j(R_{\lambda_0}(z),v,w)
=
\diff\lambda_j(v,w)
=
\lambda_j\wedge\diff\lambda_j(R_{\lambda_j}(z),v,w) >0,
\quad\forall i,j\in\{0,1\}.
\end{align*}
This readily implies that the 3-form
\begin{align*}
\lambda_t\wedge\diff\lambda_t
=
t^2\lambda_1 \wedge\diff\lambda_1 +(1-t)^2 \lambda_0\wedge\diff\lambda_0 + t(1-t) (\lambda_0\wedge\diff\lambda_1 + \lambda_1\wedge\diff\lambda_0)
\end{align*}
is a positive volume form on $Y$, and in particular each $\lambda_t$ is a contact form. We can now complete the proof by applying  a Moser trick as follows. We consider the time-dependent vector field $X_t$ on $Y$ defined by $\lambda_t(X_t)\equiv0$ and $X_t\lrcorner\,\diff\lambda_t=\lambda_0-\lambda_1$. Its flow $\psi_t:Y\to Y$, with $\psi_0=\id$, satisfies
\begin{align*}
\tfrac{\diff}{\diff t} \psi_t^*\lambda_t = \psi_t^*\big(\diff(\lambda_t(X_t)) + X_t\lrcorner\,\diff\lambda_t + \lambda_1 - \lambda_0 \big)=0,
\end{align*}
which gives the desired condition $\psi_1^*\lambda_1=\lambda_0$.
\end{proof}

\begin{proof}[Proof of Theorem~\ref{t:classification}]
Let $\lambda_1,\lambda_2$ be two Besse contact forms on a closed 3-manifold $Y$. If there exists a diffeomorphism $\psi:Y\to Y$ such that $\psi^*\lambda_2=\lambda_1$, clearly $\sigmap(Y,\lambda_1)=\sigmap(Y,\lambda_2)$. 
Conversely, assume that the two Besse closed connected contact manifolds have the same prime action spectrum $\sigmap:=\sigmap(Y,\lambda_1)=\sigmap(Y,\lambda_2)$. 
If one of the two contact forms is Zoll, then $\sigmap$ is a singleton, and the other contact form must be Zoll as well. In this case, \cite[Lemma~2.3]{Benedetti:2018ys} implies that there exists a diffeomorphism $\psi:Y\to Y$ such that $\psi^*\lambda_2=\lambda_1$. Assume now that $\lambda_1$ and $\lambda_2$ are not Zoll. By Wadsley's Theorem \cite{Wadsley:1975sp}, their prime action spectrum must have the form
\begin{align*}
\sigmap=\{\tau,\tau/a_1,...,\tau/a_s\},
\end{align*}
for some integers $s>0$ and $a_i>1$, $i=1,...,s$. Here, $\tau>0$ is the minimal common period of the Reeb orbits of both $(Y,\lambda_1)$ and $(Y,\lambda_2)$. We denote by $\Sigma_i$ the quotient of $Y$ under the locally free $\R/\tau\Z$-action defined by the Reeb flow $\phi_{\lambda_i}^t$. As we already discussed, $\Sigma_1$ and $\Sigma_2$ are orientable closed surfaces, and the quotient projections $\pi_1:Y\to\Sigma_1$ and $\pi_2:Y\to\Sigma_2$ are Seifert fibrations. 

If $Y$ is not homeomorphic to a lens space, since the two Reeb flows have the same minimal common period $\tau$, Lemmas~\ref{l:classification_Seifert} and~\ref{l:Moser_trick} imply that there exists a diffeomorphism $\psi:Y\to Y$ such that $\psi^*\lambda_2=\lambda_1$.

It remains to consider the case in which $Y$ is a lens space. Since $Y$ admits the Besse contact forms $\lambda_1$ and $\lambda_2$, it cannot be the lens space $L(0,1)$; indeed, if $Y=L(0,1)$, Theorem~\ref{t:geiges_lange}(i) would imply that the Seifert fibrations $\pi_i:Y\to\Sigma_i$ have Seifert invariants of the form $(0;\alpha,\beta,\alpha,-\beta)$, contradicting Lisca-Mati\'c's Theorem~\ref{t:lisca_matic}. Therefore, we can assume that $Y=L(p,q)$ for some $p>0$. 

We claim that the two Seifert fibrations $\pi_1:Y\to\Sigma_1$ and $\pi_2:Y\to\Sigma_2$ have the same number of singular fibers (which is at most two according to Theorem~\ref{t:geiges_lange}). Indeed, assume that one of the two fibrations, say $\pi_1:Y\to\Sigma_1$, has two singular fibers. Let $(0;\alpha_1,\beta_1,\alpha_2,\beta_2)$ be its Seifert invariants, and notice that $\alpha_1>1$ and $\alpha_2>1$. If the other Seifert fibration has only one singular fiber, then we must have $\alpha_1=\alpha_2=:\alpha$ and $\sigmap=\{\tau,\tau/\alpha\}$. By Theorem~\ref{t:geiges_lange}(iii), the quotient $n_1:=p/\alpha\in(0,p)$ is a positive integer, and we must have $p>1$ and thus $q\neq0$. This, together with Theorem~\ref{t:geiges_lange}(ii), implies that $\pi_2:Y\to\Sigma_2$ has Seifert invariants $(0;\alpha,p)$, and $\alpha\equiv q$ or $\alpha q\equiv1$ mod $p$. Therefore $(1+n_1n_2)\alpha=q$ or $(q+n_1n_2)\alpha=1$ for some $n_2\in\Z$. None of these equalities is possible: the first one since $\alpha>1$ divides $p$ and the non-zero integers $p,q$ are coprime; the latter once since $\alpha>1$. This gives a contradiction.

The Seifert fibrations $\pi_1:Y\to\Sigma_1$ and $\pi_2:Y\to\Sigma_2$ have the same Seifert invariants. Indeed, if they have only one singular fiber, then $\sigmap=\{\tau,\tau/\alpha\}$ for some integer $\alpha>1$, and Theorem~\ref{t:geiges_lange}(ii) implies that their Seifert invariants are $(0;\alpha,p)$. If they have two singular fibers, then $\sigmap=\{\tau,\tau/\alpha_1,\tau/\alpha_2\}$ for some integers $\alpha_1,\alpha_2>1$, and Theorem~\ref{t:geiges_lange}(ii) implies that their Seifert invariants are $(0;\alpha_1,b_1(p,q,\alpha_1,\alpha_2),\alpha_2,b_2(p,q,\alpha_1,\alpha_2))$. This, together with the fact that both Besse contact forms have the same minimal common period $\tau$ for their Reeb orbits, allows to apply Lemmas~\ref{l:classification_Seifert} and~\ref{l:Moser_trick}, which imply that there exists a diffeomorphism $\psi:Y\to Y$ such that $\psi^*\lambda_2=\lambda_1$.
\end{proof}

\appendix

\section{Genericity of bumpy contact forms}
\label{a:bumpy}

Let $(Y,\xi=\ker(\lambda))$ be a closed contact manifold. We recall that the contact form $\lambda$ is called \textbf{bumpy} when, for each $\tau>0$ and $z\in\fix(\phi_\lambda^\tau)$, 1 is not an eigenvalue of the linearized Poicar\'e map $\diff\phi_\lambda^\tau(z)|_\xi$. We wish to stress, here, that $\tau$ is not necessarily the minimal period of $z$. Namely, a contact form is bumpy when the simple closed orbits of its Reeb flow and all their iterates are transversally non-degenerate.
It is well known that generic contact forms supporting a given contact distribution are bumpy, see \cite[Prop.~6.1]{Hofer:1998qq}. In the proof of Lemma~\ref{l:Besse_spectral} we need a slightly stronger statement asserting that such genericity also holds when the contact form is prescribed on an embedded flow box. 

Let us recall the notion of flow box in our setting. Let $\Sigma\subset Y$ be an embedded compact ball of codimension 1 that is transverse to the Reeb vector field $R_\lambda$ and such that, for some $s>0$, the map 
\begin{align*}
 [0,s]\times\Sigma\to Y,\qquad (t,z)\mapsto\phi_\lambda^t(z)
\end{align*}
is a diffeomorphism onto its image. A flow box is a compact subset of $Y$ that is the image of one such map. 

For each compact subset $K$ of a closed manifold $Y$, we denote by $C^r_K(Y)$ the space of $C^r$ functions $f:Y\to\R$ such that $f|_{K}\equiv0$; in the following, $C^r_K(Y)$ will be endowed with the $C^r$ topology.

\begin{prop}\label{p:bumpy}
Let $(Y,\xi=\ker(\lambda))$ be a closed contact manifold, $K\subsetneq Y$ a flow box for the Reeb flow $\phi_\lambda^t$, and $2\leq r\leq\infty$. Then, there is a $G_\delta$-dense subset $\B\subset C^r_K(Y)$ such that, for each $b\in\B$, the contact form $e^b\lambda$ is bumpy.
\end{prop}

The proof of this proposition is analogous to the one provided by Anosov \cite{Anosov:1982ix} in the case of geodesic flows, and we will carry it over after some preliminaries. 
For each $2\leq r\leq\infty$ and $T_1\geq T_0>0$, we denote by $\B^r(T_0,T_1)$ the subset of those $b\in C^r_K(Y)$ such that 
\begin{align*}
1\not \in\sigma(\diff\phi_{e^b\lambda}^{k\tau}(z)|_\xi),
\qquad
\forall 
\tau\in(0,T_0],\ k\in \N\cap(0,T_1/\tau],\ z\in \fix(\phi_{e^b\lambda}^{\tau}).
\end{align*}
Here, as usual, $\sigma(\cdot)$ denotes the spectrum of a linear endomorphism.

\begin{lem}
The subset $\B^r(T_0,T_1)$ is open in $C^r_K(Y)$. 
\end{lem}

\begin{proof}
Assume that a function $b\in C^r_K(Y)$ does not belong to the interior of $\B^r(T_0,T_1)$, so that there exists a sequence $b_n\in C^r_K(Y)\setminus\B^r(T_0,T_1)$ converging to $b$. Therefore, there exist sequences $\tau_n\in(0,T_0]$, $k_n\in\N\cap(0,T_1/\tau_n]$, and $z_n\in\fix(\phi_{e^{b_n}\lambda}^{\tau_n})$ such that $1\in\sigma(\diff\phi_{e^{b_n}\lambda}^{k_n\tau_n}(z_n)|_\xi)$. Up to passing to appropriate subsequences, we can assume that $\tau_n\to\tau\in(0,T_0]$, $k_n\equiv k\in(0,T_1/\tau]$, and $z_n\to z$. However, this implies that $\phi_{e^b\lambda}^{\tau}(z)=z$ and $1\in\sigma(\diff\phi_{e^{b}\lambda}^{k\tau}(z)|_\xi)$, and thus $b\not\in\B^r(T_0,T_1)$.
\end{proof}

We set $\B^r(T):=\B^r(T,T)$. Namely, $\B^r(T)$ is the set of those $b\in C^r_K(Y)$ such that all the (possibly iterated) closed orbits of the Reeb flow $\phi_\lambda^t$ with period at most $T$ are transversely non-degenerate. 
We introduce the $C^{r-1}$ map
\begin{align*}
\Phi:C^r_K(Y)\times Y\times(0,\infty)\to Y\times Y,
\qquad
\Phi(b,z,t)=(z,\phi_{e^b\lambda}^t(z)),
\end{align*}
and we denote by $\Phi_b:=\Phi(b,\cdot,\cdot):Y\times (0,\infty)\to Y\times Y$ its restrictions. Notice that, if $\tau>0$ and $z\in\fix(\phi_{e^b\lambda}^\tau)$, then $1\not \in\sigma(\diff\phi_{e^b\lambda}^{\tau}(z)|_\xi)$ if and only if the image of $\diff\Phi_b(z,\tau)$ is transverse to $\Tan_{(z,z)}\Delta$, where $\Delta\subset Y\times Y$ is the diagonal submanifold. Even when this transversality condition is not satisfied, we still have the following one. From now on, we assume that
\[3\leq r<\infty,\]
so that the map $\Phi$ is at least $C^2$.

\begin{lem}\label{l:transversality}
If $\tau>0$ is the minimal period of a closed Reeb orbit $\phi_{e^b\lambda}^t(z)$, then the image of $\diff\Phi(b,z,\tau)$ is transverse to $\Tan_{(z,z)}\Delta$. In particular, for each $T>0$, the map $\Phi|_{\B^r(T,2T)\times Y\times(0,2T)}$ is transverse to the diagonal $\Delta\subset Y\times Y$.
\end{lem}

\begin{proof}
For each $H\in C^\infty(Y)$, we denote by $\psi_H^t:Y\to Y$ the contact isotopy generated by the contact Hamiltonian vector field $X_H$, which is defined by \[e^b\lambda(X_H)=H,\qquad X_H\lrcorner\,\diff(e^b\lambda)=-\diff H+\diff H(R_{e^b\lambda})e^b\lambda.\]
Notice that $\phi_{e^b\lambda}^t=\psi_1^t$. Assume that $\tau>0$ is the minimal period of a closed Reeb orbit $t\mapsto\psi_1^t(z)$. Since $K$ is a flow box for the Reeb flow $\psi_1^t(z)$, this closed orbit must intersect its complement $Y\setminus K$. Let $t_0\in[0,\tau)$ be such that $z_0:=\psi_1^{t_0}(z)\in Y\setminus K$. For each $v\in\Tan_z Y$ and for each open neighborhood $U\subset Y\setminus K$ of $z_0$, we can find a family of smooth functions $H_s:Y\to\R$ smoothly depending on $s\in(-\epsilon,\epsilon)$ such that $H_0\equiv1$, $H_s|_{Y\setminus U}\equiv 1$ for each $s\in(-\epsilon,\epsilon)$, and $\tfrac{\diff}{\diff s}\big|_{s=0}\psi_{H_s}^\tau(z)=v$. We set $b_s:=-\log(H_s)$, and notice that $b_0\equiv0$ and $b_s|_{Y\setminus U}\equiv 0$ for all $s\in(-\epsilon,\epsilon)$. In particular each $b_s$ belongs to $C^r_K(Y)$. The contact Hamiltonian vector field $X_{H_s}$ is the Reeb vector field associated to the contact form $e^{b+b_s}\lambda=\tfrac1{H_s}e^b\lambda$, i.e., $\psi_{H_s}^t=\phi_{e^{b+b_s}\lambda}^t$. Therefore, if we set $b':=\partial_s b_s|_{s=0}$, we have
\begin{align*}
\diff\Phi(b,z,\tau)(b',0,0)=\tfrac{\diff}{\diff s}\big|_{s=0}\Phi(b+b_s,z,\tau)=
\tfrac{\diff}{\diff s}\big|_{s=0} (z,\psi_{H_s}^\tau(z))
=(0,v).
\end{align*}
This readily implies that $\diff\Phi(b,z,\tau)$ is transverse to $\Tan_{(z,z)}\Delta$.
\end{proof}

\begin{lem}\label{l:dense1}
For each $T>0$ and $S\in(T,2T)$, the intersection $\B^r(S)\cap\B^r(T,2T)$ is dense in $\B^r(T,2T)$.
\end{lem}

\begin{proof}
By Lemma~\ref{l:transversality}, the $C^2$ map $\Phi|_{\B^r(T,2T)\times Y\times(0,2T)}$ is transverse to the diagonal $\Delta\subset Y\times Y$. Therefore, by Abraham's infinite dimensional transversality theorem \cite{Abraham:1967wj}, there exists a dense subset $\B\subset\B^r(T,2T)$ such that, for each $b\in\B$, the map $\Phi_b|_{Y\times(0,2T)}$ is transverse to the diagonal $\Delta\subset Y\times Y$. This implies that $\B$ is contained in $\B^r(S)$ for each $S\in(T,2T)$, and the lemma follows.
\end{proof}

\begin{lem}\label{l:dense2}
For each $T>0$, $\B^r(T,2T)$ is dense in $\B^r(T)$.
\end{lem}

\begin{proof}
Consider an arbitrary $b_0\in\B^r(T)$. Notice that the subset
\begin{align*}
F:=\bigcup_{\tau\in(0,T]} \fix(\phi_{e^{b_0}\lambda}^\tau)
\end{align*}
is the union of finitely many closed orbits $\gamma_1,...,\gamma_k$ of the Reeb flow $\phi_{e^{b_0}\lambda}^t$. Since $b_0\equiv0$ on the flow box $K$, each $\gamma_i$ intersects the open set $Y\setminus K$. Since all the closed orbits of $\phi_{e^{b_0}\lambda}^t$ with period at most $T$ are transversally non-degenerate, there exists $\epsilon>0$ such that no closed orbit of $\phi_{e^{b_0}\lambda}^t$ has minimal period in the interval $(T,T+\epsilon]$.
For each $i=1,...,k$, we choose a point $z_i(b_0)$ on the intersection of the closed orbit $\gamma_i$ with $Y\setminus K$, and we denote by $\tau_i(b_0)\in(0,T]$ the minimal period of $\gamma_i$. Since all the $\gamma_i$'s are non-degenerate $\tau_i(b_0)$-periodic orbits, there exist an open neighborhood $\UU\subset C^r_K(Y)$ of $b_0$, open neighborhoods $U_i\subset Y\setminus K$ of $z_i(b_0)$, and continuous (indeed, even more regular) maps
$\bm z:\UU\to U_1\times ...\times U_k$,
$\bm z(b) = (z_1(b),...,z_k(b))$ and $\bm\tau:\UU\to (0,T+\epsilon/2]\times ...\times (0,T+\epsilon/2]$, $\bm\tau(b) = (\tau_1(b),...,\tau_k(b))$ with the following properties: for each $b\in\UU$, the only closed orbits of the Reeb flow of $e^b\lambda$ with minimal period at most $T+\epsilon/2$ and intersecting $U_1\cup...\cup U_k$ are the $\phi_{e^b\lambda}^t(z_i(b))$'s; the minimal period of $\phi_{e^b\lambda}^t(z_i(b))$ is $\tau_i(b)$, and $1\not\in\sigma(\diff\phi_{e^b\lambda}^{\tau_i(b)}(z_i(b)))$. We can choose a smaller open neighborhood $\WW\subset\UU$ of $b_0$ so that, for each $b\in\WW$, the Reeb flow of $e^b\lambda$ does not have closed orbits with period at most $T+\epsilon/2$ not intersecting $U_1\cup...\cup U_k$. For every such open neighborhood $\WW$, it is well known that we can choose $b\in \WW$ such that, for each $i=1,...,k$ and $h\in\N\cap(0,2T/\tau_i(b))$, we have $1\not\in\sigma(\diff\phi_{e^b\lambda}^{h\tau_i(b)}(z_i(b)))$. Therefore such $b$ belongs to $\B^r(T,2T)$.
\end{proof}

\begin{proof}[Proof of Proposition~\ref{p:bumpy}]
We define the $G_\delta$-set
\begin{align*}
\B^r:=\bigcap_{S\in\N} \B^r(S).
\end{align*}
We first assume that $3\leq r<\infty$.
Lemmas~\ref{l:dense1} and~\ref{l:dense2} imply that the open subset $\B^r(\tfrac32 T)$ is dense in $\B^r(T)$. Therefore, $\B^r(S)$ is dense in $\B^r(T)$ for all $S\geq T$, and by the Baire category theorem we conclude that $\B^r$ is dense in $\B^r(T)$ for each $T>0$. Notice that, for each $b\in C^r_K(Y)$, there exists $T>0$ such that no closed orbit of the Reeb vector field $\phi_{e^b\lambda}^t$ has period less than or equal to $T$. In particular, every $b\in C^r_K(Y)$ is contained in $\B^r(T)$ for some $T>0$. This, together with the above density argument, implies that $\B^r$ is dense in $C^r_K(Y)$.
Since $C^r_K(Y)\hookrightarrow C^2_K(Y)$ is a dense inclusion, we readily infer that the set $\B^2$ is dense in $C^2_K(Y)$. 

The set $\B^\infty(S)=\B^r(S)\cap C^\infty_K(Y)$ is open in $C^\infty_K(Y)$. For $2\leq r<\infty$, since $\B^r(S)$ is open in $C^r_K(Y)$, and both $\B^r(S)$ and $C^\infty_K(Y)$ are dense in $C^r_K(Y)$, we infer that $\B^\infty(S)$ is dense in $C^r_K(Y)$. We recall that the topology of $C^\infty_K(Y)$ is generated by open sets of the form $W\cap C^\infty_K(Y)$, where $W$ is an open subset of some $C^r_K(Y)$ with $2\leq r<\infty$. If $W\subseteq C^r_K(Y)$ is one such non-empty open subset, we have
\[\B^\infty(S)\cap (W\cap C^\infty_K(Y))=\B^\infty(S)\cap  W\neq\varnothing.\] 
This shows that $\B^\infty(S)$ is open and dense in $C^\infty_K(Y)$. By the Baire category theorem, $\B^\infty$ is dense in $C^\infty_K(Y)$.
\end{proof}

\bibliography{_biblio}
\bibliographystyle{amsalpha}

\end{document}